\documentclass[12pt]{article}
\usepackage{amssymb,amsmath, graphicx, url, xypic}
\usepackage{amsthm}

\newcommand{\be}{\begin{enumerate}}
\newcommand{\ee}{\end{enumerate}}
\newcommand{\beq}{\begin{equation}}
\newcommand{\eeq}{\end{equation}}
\newcommand{\bea}{\begin{eqnarray}}
\newcommand{\eea}{\end{eqnarray}}
\newcommand{\beas}{\begin{eqnarray*}}
\newcommand{\eeas}{\end{eqnarray*}}

\newcommand{\sn}{\mathfrak{S}_n}
\newcommand{\covby}{\lessdot}

\def\({\left(}
\def\){\right)}

\swapnumbers
\theoremstyle{plain}
\newtheorem{theorem}{Theorem}[section]
\newtheorem{proposition}[theorem]{Proposition}
\newtheorem{lemma}[theorem]{Lemma}
\newtheorem{corollary}[theorem]{Corollary}

\theoremstyle{definition}

\newtheorem{example}[theorem]{Example}
\theoremstyle{remark}

\theoremstyle{remark}

\begin{document}

\title{
\bf The Weak Bruhat Order and Separable Permutations
\footnote{This research was carried out under the
    direction of R. Stanley when the author was an undergraduate at
    M.I.T.} 
    }
    \author{Fan Wei\footnote{fan\_wei@mit.edu, MIT}}
    \date{September 21, 2010}
\maketitle{\small \textsc{\bf Abstract}
In this paper we consider the rank generating function of a
separable permutation $\pi$ in the weak Bruhat order on the two
intervals $[\text{id}, \pi]$ and $[\pi, w_0]$, where $w_0 =
n,(n-1),\dots, 1$. We 
show a surprising result that the product of these two
generating functions is the generating function for the
symmetric group with the weak order. We then
obtain explicit formulas for the rank generating functions
on $[\text{id}, \pi]$ and $[\pi, w_0]$, which leads to
the rank-symmetry and unimodality of the two graded posets. }

\vskip 5mm

\section{Introduction and Definitions} \label{sec1}
Let $\sn$ denote the symmetric group of all permutations of
$1,2,\dots,n$. Define the \emph{length} of the permutation
$\pi=a_1a_2\cdots a_n\in\sn$ by $$ \ell(\pi)=\#\{ 1\leq i<j\leq
n\colon a_i>a_j\}, $$ which is the number of inversions of $\pi$. One
of the fundamental partial orderings of $\sn$ is the \emph{weak
  (Bruhat) order}. A cover relation $\pi\covby \sigma$ in weak order,
i.e., $\pi<\sigma$ and nothing is in between, is defined by
$\sigma=\pi s_i$ for some adjacent transposition $s_i=(i,i+1)$,
provided that $\ell(\sigma)> \ell(\pi)$.  We are multiplying
permutations right-to-left, so for instance $2413s_2=2143$. The weak
order makes $\sn$ into a graded poset of rank $\binom n2$. If
$\pi=a_1a_2\cdots a_n\in\sn$, then the rank function of $\sn$ (which
will have the weak order unless stated otherwise) is the function
$\ell$. The rank generating function is then given by 
$$ F(\sn,q) =\sum_{\pi\in\sn}q^{\ell(\pi)} = [n]!, $$     
where $[n]!=[1][2]\cdots[n]$ and $[i]=1+q+q^2+\cdots+q^{i-1}$.

A permutation $\pi=a_1a_2\cdots a_n\in\sn$ is \emph{3142-avoiding} and
\emph{2413-avoiding} if there do not exist $i<j<k<h$ with
$a_{j}<a_{h}<a_{i}<a_{k}$ or $a_{k}<a_{i}<a_{h}<a_{j}$. Such
permutations are also called \emph{separable}. For a general
introduction to pattern avoidance, see \cite{bona}. Separable
permutations first arose in the work of Avis and Newborn
\cite{av-ne} and have subsequently received a lot of attention. A
survey of some of their properties appears in \cite{albert}. In
particular, the number of separable permutations in $\sn$ is the
(large) Schr\"oder number $r_{n -1}$. Let id denote the identity element
of $\sn$ (the unique minimal element in weak order), and let
$w_0=n,n-1,\dots,1$, the unique maximal element. For $\pi\in\sn$,
let $\Lambda_\pi$ denote the interval $[\mathrm{id},\pi]$ (in weak
order), and let $V_\pi=[\pi,w_0]$. Thus $\Lambda_\pi$ and $V_\pi$ are
themselves graded posets (with rank$(\pi)=0$ in $V_\pi$). The main
result of this paper is the surprising formula
\beq F(\Lambda_\pi,q)F(V_\pi,q) = F(\sn,q) =[n]!. \label{eqmain} \eeq

Equation~\eqref{eqmain} was conjectured by R. Stanley. It was inspired
by an observation of Steven Sam that if $\pi$ is 231-avoiding, then
$\Lambda_\pi$ appears to be rank-symmetric and rank-unimodal. These
two properties are simple consequences of Theorem \ref{formula}. (See
Corollary \ref{sym and unim}.) Figure~\ref{fig1} shows the Hasse
diagram of $\mathfrak{S}_4$. If for instance $\pi=4132$ (which is
separable), then
$F(\Lambda_\pi,q)=1+2q+2q^2+2q^3+q^4$ and $F(V_\pi,q)=1+q+q^2$. Then
multiplying $F(\Lambda_\pi,q)$ and $F(V_\pi,q)$ gives us $[4]!$.

We also give a convenient method to find an explicit formula for
$F(\Lambda_\pi,q)$ and $F(V_\pi,q)$. In fact, when $\pi=a_1a_2\cdots
a_n\in\sn$ is 231-avoiding, meaning that there do not exist $i<j<k$
with $a_{k}<a_{i}<a_{j}$, the explicit formula for
$F(\Lambda_\pi,q)$ is given by
\beq F(\Lambda_\pi,q) = \prod_{i=1}^n [c_i], \eeq 
where $a_{c_i+i}$ is
the first element to the right of $a_i$ in $\pi$ satisfying $a_{c_i+i}
> a_i$, setting $a_{n+1}=\infty$.

$$\xymatrix@R=18pt@C=-5pt@M=0.5pt@H=1pt@W=1pt{
& & & & & 4321 \ar@{-}[dll] \ar@{-}[d] \ar@{-}[drr] & & & & &\\
& & & 4312 \ar@{-}[dll] \ar@{-}[drr] & & 4231 \ar@{-}[dll] \ar@{-}[drr] & & 3421 \ar@{-}[dll] \ar@{-}[drr] & & & \\
& 4132 \ar@{-}[dl] \ar@{-}[dr] & & 4213 \ar@{-}[dl] \ar@{-}[dr] & & 3412 \ar@{-}[dr] & & 2431 \ar@{-}[dlll] \ar@{-}[dr] & & 3241\ar@{-}[dl] \ar@{-}[dr] & \\
1432 \ar@{-}[dr] \ar@{-}[drrr] & & 4123 \ar@{-}[dl] & & 2413 \ar@{-}[dr] & & 3142 \ar@{-}[dlll] \ar@{-}[dr] & & 3214 \ar@{-}[dl] \ar@{-}[dr] & & 2341 \ar@{-}[dl] \\
& 1423 \ar@{-}[drr]  & & 1342 \ar@{-}[drr] & & 2143 \ar@{-}[dll] \ar@{-}[drr] & & 3124 \ar@{-}[dll] & & 2314\ar@{-}[dll] & \\
& & & 1243 \ar@{-}[drr]  & & 1324 \ar@{-}[d]  & & 2134 \ar@{-}[dll]  & & & \\
& & & & & 1234 & & & & &\\ }\label{fig1}$$ 
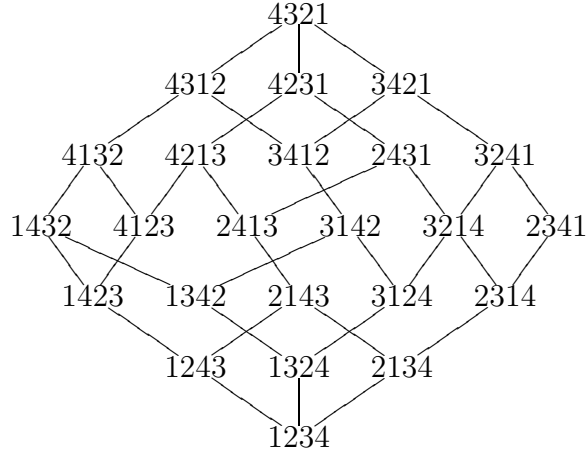
\begin{center}{Figure 1. The graded poset $\mathfrak{S}_4$ under
    weak order} \end{center} 
\vspace{0.5cm}

The \emph{inversion poset} $P_\pi$ of $\pi = a_1a_2\cdots a_n\in \sn$
has the relations $a_i<a_j$ in $P$ if $i<j$ and $a_i<a_j$ in
$\mathbb{Z}$. Figure 2 is the diagram of the inversion posets of
the permutations $34125$ and $31425$.

$$\xymatrix@R=18pt@C=8pt@M=3pt@H=1pt@W=1pt{
& & & 5 \ar@{-}[dl]\ar@{-}[dll] & & & & &  5 \ar@{-}[dl]\ar@{-}[dll]\\
& 4 \ar@{-}[dl] & 2 \ar@{-}[dl] & & & & 4 \ar@{-}[d] \ar@{-}[dl] & 2 \ar@{-}[dl] &  \\
3 & 1 & & & & 3 & 1 & & }$$
\label{fig2}
\begin{center}
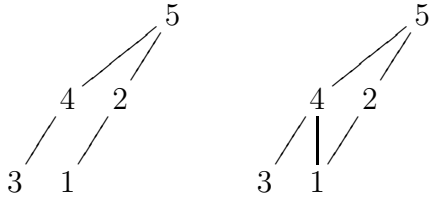
{Figure 2.  The inversion posets of 34125 (left) and 31425 (right)}
\end{center}

\vspace{.3cm}
Let $P$ and $Q$ be posets on disjoint sets. The \emph{disjoint union}
$P+Q$ is the poset on the union $P\cup Q$ such that $s\leq t$ in $P+Q$
if either $s,t\in P$ and $s\leq t$ in $P$, or $s,t\in Q$ and $s\leq t$
in $Q$.  The \emph{ordinal sum} $P\oplus Q$ is the poset on the union
$P\cup Q$ such that $s\leq t$ in $P\oplus Q$ if either $s,t\in P$ and
$s\leq t$ in $P$, or $s,t\in Q$ and $s\leq t$ in $Q$, or $s\in P$ and
$t\in Q$.

The following lemma is easy to prove, so we omit the proof here.

\begin{lemma} \label{sum}
Let $\pi \in \sn$ with $\pi = \pi_A \pi_B$, where $\pi_A$ is a
permutation of size $m$ and $\pi_B$ is a permutation of size $n-m$ for
some $m < n$. Then  
\begin{itemize}
\item $P_\pi = P_{\pi_A} + P_{\pi_B}$ if and only if $\pi_B$ is a
  permutation of the letters $\{1,2,\dots, m\}$ and $\pi_A$ is a
  permutation of the letters $\{m+1, m+2, \dots, n\}$. 
\item $P_\pi = P_{\pi_A} \oplus P_{\pi_B}$ if and only if $\pi_A$ is a
  permutation of the letters $\{1,2,\dots, m\}$ and $\pi_B$ is a
  permutation of the letters $\{m+1, m+2, \dots, n\}$. 
\end{itemize}
\end{lemma}
\vspace{0.2cm}
A \emph{linear extension} of a poset $P$ on the set $\{1,2,\dots,n\}$
is a permutation $\pi=a_1\cdots a_n\in\sn$ such that if $i<j$ in $P$,
then $i$ precedes $j$ in $\pi$. 
We use $\mathcal{L}(P)$ to denote the set of linear extensions of
$P$. Since a linear extension $\pi$ of a poset $P$ on $\{1,\dots,n\}$
has been defined as a permutation of $\{1,\dots,n\}$, it has length
$\ell(\pi)$ as defined above. We define
  $$ F(\mathcal{L}(P), q) = \sum_{\pi \in
    \mathcal{L}(P)}q^{\ell(\pi)}. $$ 

We have the following rules for the operation on $F(\mathcal{L}(P), q).$
\begin{lemma}\label{op}
Let $P$ and $Q$ be two posets, where $P$ is on $\{1,2,\dots,
m\}$ and $Q$ is on $\{m+1, \dots, m+n\}$. Then  
\beq F(\mathcal{L}(P \oplus Q), q) = F(\mathcal{L}(P,
q))F(\mathcal{L}(Q, q) ) \label{opsum},\eeq 
\beq F(\mathcal{L}(P + Q), q) = F(\mathcal{L}(P, q))F(\mathcal{L}(Q,
q)) \begin{bmatrix}m+n\\ m \end{bmatrix}\label{opunion}, \eeq   
where $\begin{bmatrix}m+n\\ m \end{bmatrix} = \dfrac{[m+n]!}{[m]![n]!}.$
\end{lemma}

The proof of (\ref{opsum}) is immediate by considering the definition
of ordinal sum and counting the number of inversions. The proof of
(\ref{opunion}) follows from the theory of $P$-partitions,
a straightforward extension of the second proof of Proposition 1.3.17
of \cite{EC1}. 

A \emph{reduced decomposition} of a permutation $\pi\in\sn$ is a
sequence $(i_1,i_2,\dots,i_\ell)$ such that $\pi =s_{i_1}s_{i_2}\cdots
s_{i_\ell}$ and $\ell$ is minimal, viz., $\ell=\ell(\pi)$. If
$\pi=\pi_0\covby \pi_1\covby\cdots\covby \pi_m= \sigma$ is a saturated
chain $C$ from $\pi$ to $\sigma$, where $\pi_j=\pi_{j-1}s_{i_j}$, then
$r(C):=(i_1,\dots,i_\ell)$ is a reduced decomposition of
$\pi^{-1}\sigma$. 
Write $R(\pi)$ for the set of reduced decompositions of $\pi$. Thus
the map $C\mapsto r(C)$ is a bijection between saturated chains from
$\mathrm{id}$ to $\pi$ and reduced decompositions of $\pi$. 

With the definitions above, we proceed to the proofs of the main
theorem and the explicit formula for $F(\Lambda_\pi,q)$. 


\section{Preliminary Results}
The following lemma states a property of separable permutations
which is of great importance to our proof of the main theorem. 
\begin{lemma}\label{separable}
If $n>1$ and $\pi=a_1a_2\cdots a_n\in \sn$ is a separable permutation,
then we can write $\pi=\pi_A\pi_B$ (concatenation of words), where
$\pi_A$ and $\pi_B$ are both separable permutations satisfying one of
the two following properties:    
\begin{itemize}
\item$\pi_A$ is a permutation of $1,2,\dots, m$ and $\pi_B$ is of
  $m+1,\dots, n$ for some $m$ with $1 \leq m < n$; 
\item$\pi_A$ is a permutation of $m+1,\dots, n$ and $\pi_B$ is of
  $1,2,\dots, m$ for some $m$ with $1 \leq m < n$. 
\end{itemize}
\end{lemma}
Lemma \ref{separable} is well-known and easy to prove; thus we omit
the proof here.  
The following lemma is an immediate consequence of Lemma \ref{separable}

\begin{corollary} \label{cor:ppi}
If $n>1$ and $\pi=a_1a_2\cdots a_n\in \sn$ is a separable permutation,
then there exist two disjoint nonempty posets $P_{\pi_A}, P_{\pi_B}$
such that $P_\pi = P_{\pi_A} + P_{\pi_B}$ or $P_\pi = P_{\pi_A} \oplus 
P_{\pi_B}$. 
\end{corollary}

The following lemma is a special case of a result of Bj\"orner and
Wachs \cite[Thm.~6.8]{bj-wa}.

\begin{lemma} \label{FF}
Let $\pi$ be any permutation in $\sn$, then $F(\mathcal{L}(P_\pi), q) = F(\Lambda_\pi, q)$.
\end{lemma}

\vskip 0.2cm 
Now we arrive at one of the main preliminary results of this section.

\begin{proposition} \label{below}  
If $\pi =a_1a_2\cdots a_n \in \mathfrak{S}_n$ is a separable
permutation, then the following hold: 
\paragraph{(i)} When $a_1 < a_n$, we can write $\pi = \pi_A\pi_B$
where $\pi_A$ is a permutation of size $m$ for some $m$ with $1 \leq m
< n$, and 
      \beq F(\Lambda_\pi, q) = F(\Lambda_{\pi_A}, q) \cdot F(\Lambda_{\pi_B}, q) \label{belowsb}. \eeq
\paragraph{(ii)} When $a_1 > a_n$, we can write $\pi = \pi_A
\pi_B$, where $\pi_A$ is a permutation of size $m$ for some $m$ with
$1 \leq m < n$, and \beq F(\Lambda_\pi, q) = \begin{bmatrix}n\\ m
\end{bmatrix} F(\Lambda_{\pi_A}, q) \cdot F(\Lambda_{\pi_B},
q). \label{belowbs} \eeq 
\end{proposition}

\begin{proof}
Let $P$ be the inversion poset of $\pi$, $P_{A}$ be the inversion
poset of $\pi_A$, and $P_{B}$ be the inversion poset of $\pi_B$. 

When $a_1 < a_n$, it follows from Lemma \ref{separable} that we can
write $\pi = \pi_A \pi_B$, where $\pi_A$ is a permutation of $\{1,2,
\dots, m\}$ and $\pi_B$ is a permutation of $\{m+1, m+2, \dots,
n\}$. By Lemma \ref{sum}, we have $P = P_{A} \oplus P_{B}$. It
follows from Lemma \ref{op} that $$F(\mathcal{L}(P_{A} \oplus
P_{B}), q) = F(\mathcal{L}(P_{A}), q)F(\mathcal{L}(P_{B}),
q). $$ 

Since $\pi, \pi_A,\pi_B$ are all separable permutations, by Lemma
\ref{FF}, we have \begin{eqnarray*} 
 F(\Lambda_\pi, q) = F(\mathcal{L}(P, q)) &=& F(\mathcal{L}(P_{A}
 \oplus P_{B}), q) \\ &= &F(\mathcal{L}(P_{A}),
 q)F(\mathcal{L}(P_{B}), q)\\ &= &F(\Lambda_{\pi_A}, q)
 F(\Lambda_{\pi_B}, q). \end{eqnarray*}

The proof of (ii) is similar.
\end{proof}

\begin{proposition}\label{above}   
If $\pi =a_1a_2\cdots a_n \in \mathfrak{S}_n$ is a separable
permutation, then the following hold: 
\paragraph{(i)} If $a_1 < a_n $, then we can write $\pi = \pi_A\pi_B$
where $\pi_A$ is a permutation of size $m$ for some $m$ with $1 \leq m
< n$, and  
 \beq F(V_{\pi}, q) =\begin{bmatrix}n\\ m \end{bmatrix} F(V_{\pi_A},
 q) \cdot F(V_{\pi_B}, q) \label{abovesb}.\eeq 
\paragraph{(ii)} If $a_1 > a_n$, then we can write $\pi = \pi_A
\pi_B$, where $\pi_A$ is a permutation of size $m$ for some $m$ with
$1 \leq m < n$, and \beq F(V_\pi, q) = F(V_{\pi_A}, q) \cdot
F(V_{\pi_B}, q) \label{abovebs}.\eeq 
\end{proposition}

The proof is similar to that of Proposition \ref{below} by using the
\emph{complement} $\pi^c$ of a permutation $\pi = a_1a_2\cdots a_n \in
\sn$ defined by $\pi^c = a_1'a_2'\cdots a_n'$ where $a_i' = n+1-a_i$
for all $1 \leq i \leq n$. 

A standard property of $\pi^c$ and weak order is stated in the following lemma, and we omit the easy proof here. 
\begin{lemma} \label{up bij}
The rank relation between a permutation and its complement is given by
$$\ell(\pi^c) = \binom n2 - \ell(\pi).$$ 
In fact, there exists a bijection $\mu: [\pi, w_0] \to [\mathrm{id},
\pi^c]$ defined by $\mu(w) = w^c$ for all $w \in [\pi, w_0]$. 
\end{lemma}
%

\vspace{0.2cm}

\begin{proof}[Proof of Proposition \ref{above}]
For any $\omega \in [\pi, w_0]$, by Lemma \ref{up bij} and the fact
that 
 $$\ell(w^c) = \binom n2 - \ell(w) = \ell(\pi^{-1}w_0) -
    \ell(\pi^{-1}w),$$ 
any $q^{\ell(\pi^{-1}\omega)}$ in $F(V_\pi, q)$
corresponds uniquely to a term
$q^{\ell(\pi^{-1}w_0)-{\ell(\pi^{-1}\omega)}}$ in $F(\Lambda_{\pi^c},
q)$. Thus  
\beq q^{\ell(\pi^{-1}w_0)} F(V_\pi, q^{-1}) = F(\Lambda_{\pi^c}, q).
\label{UD} \eeq 

We now consider $\pi^c$ in the two cases in Proposition \ref{above}.

\vspace{0.15cm}
\textbf{(i)} When $a_1 < a_n $, by equation (\ref{belowbs}) we have  
\beq F(\Lambda_{\pi^c}, q) = \begin{bmatrix}n\\ m \end{bmatrix}
F(\Lambda_{{\pi_A}^c}, q) \cdot F(\Lambda_{{\pi_B}^c},
q)\label{1}.\eeq  

Combining (\ref{UD}) and  (\ref{1}) gives us  
\beq q^{\ell(\pi^{-1}w_0)} F(V_\pi, q^{-1}) = \begin{bmatrix}n\\ m
\end{bmatrix}F(V_{{\pi_A}^c}, q^{-1}) F(V_{{\pi_B}^c}, 
 q^{-1}) \cdot q^{\binom m2 - \ell(\pi_A)}\cdot q^{\binom{n-m}{2}
   - \ell(\pi_B)}\label{2}. \eeq   
Since the letters in $\pi_A$ are all smaller than the letters in
$\pi_B$, we have 
$\ell(\pi) = \ell(\pi_A) + \ell(\pi_B)$. Substituting $q^{-1}$ for
$q$ in (\ref{2}), which converts $\begin{bmatrix}n\\ m \end{bmatrix}$
into $q^{\binom n2-\binom m2 -\binom{n-m}{2}}\begin{bmatrix}n\\ m
\end{bmatrix}$, completes the proof of 
(\ref{abovesb}). 

\vspace{0.17cm}
\textbf{(ii)} Since all the letters in $\pi_A$ are greater than the
letters in $\pi_B$, we have $$\ell(\pi) = (n-m)m + \ell(\pi_A) +
\ell(\pi_B).$$ The rest of (\ref{abovebs}) can be proved analogously.   
\end{proof} 


\section{Main Results}
\subsection{Main Theorem}
\begin{theorem}  \label{main theorem}
Let $\pi\in \mathfrak{S}_n$, $\Lambda_\pi = [\mathrm{id}, \pi]$, and
$V_\pi = [\pi, w_0]$. The following equation holds for any separable
permutation $\pi$: 
\beq F(\Lambda_\pi, q)F(V_\pi. q) = F(\mathfrak{S}_n, q) = [n]!.\eeq
\end{theorem}

\begin{proof}
When $n=2$, it is easy to verify that the expression holds. Suppose
the statement holds when $k < n$ for some $n \geq 3$; we want to show
that when $k = n$, the statement still holds.

Let $\pi_A$ and $\pi_B$ be the same as before.
When $a_1> a_n$  we have by (\ref{belowbs}) and (\ref{abovebs}) that
$$
F(\Lambda_\pi, q)F(V_\pi. q) = \begin{bmatrix}n\\ m 
\end{bmatrix}F(\Lambda_{\pi_A}, q)F(\Lambda_{\pi_A}, q)\cdot
F(V_{\pi_A}, q)F(V_{\pi_B}, q).$$ Thus by the inductive hypothesis, we
have $$ F(\Lambda_\pi, q)F(V_\pi, q) = \begin{bmatrix}n\\ m
\end{bmatrix}[m]! [n-m]! = [n]!.$$ 

The proof for $a_1< a_n$ is similar.  
\end{proof}

\subsection{A Bijection $\varphi\colon \Lambda_w\times V_w\to \sn$}
We can also give a  bijective proof of Theorem \ref{main
  theorem}. 
\begin{theorem} \label{bijj}
Let $\pi = a_1a_2 \cdots a_n \in \mathfrak{S}_n$ be a separable
permutation. The map  
$$ \phi: \Lambda_\pi \times V_\pi \rightarrow
\sn $$ defined by $\phi (u,v)= u^{-1}v$, where $u\leq \pi$ and $v\geq
\pi$, is a bijection. 
\end{theorem}

Since $(u^{-1}v)^{-1} = v^{-1}u$, it is a direct consequence of
Theorem \ref{bijj} that the map $$ \phi': \Lambda_\pi \times V_\pi
\rightarrow \sn $$ defined by $\phi'(u,v) = v^{-1}u$ for $u\leq \pi$
and $v\geq \pi$ is also a bijection.

We use the following lemma to prove this theorem.
\begin{lemma} \label{no interaction}
{If $\pi = a_1a_2\cdots a_n \in\sn$ is a separable permutation with
  $a_1<a_n$, and $(i_1,\dots,i_\ell)\in R(\pi)$, then there exists an
  integer $m$ with $1 \leq m < n$ such 
that none of the simple transpositions $s_{i_j}$ transposes an element
in $A_\pi = \{1,2,\dots, m\}$ with an element in $B_\pi = \{m+1,
\dots, n\}$. In other words, there is no interaction between the sets
$A_\pi$ and $B_\pi$.} 
\end{lemma}

The proof of Lemma \ref{no interaction} can be achieved easily from
the definition of weak order. 
\begin{proof}

If the lemma does not hold, then in the sequence of all simple
transpositions there exists a nonempty subsequence consisting of
simple transpositions between the letters in $A_\pi$ and the letters
in $B_\pi$. Suppose the last transposition in this subsequence is
between $a \in A_\pi$ and $b \in B_\pi$. 
From Proposition \ref{separable} we know that $a$ is to the left of
$b$. Since 
$a < b$, by the definition of weak order the permutation after the
transposition is covered by the permutation before swapping $a$ and $b$, which leads to a contradiction. 
\end{proof}

\vspace{0.2cm}
\begin{proof}[Proof of Theorem \ref{bijj}]
When $a_1 < a_n$, by Lemma \ref{separable} we can write $\pi = \pi_A
\pi_B$ where $\pi_A$ is a separable permutation of $\{1,2,\dots, m\}$
for some $m > 0$. 

For the injectivity part, we want to show that there do not exist two
different pairs $(u_1, v_1), (u_2, v_2) \in \Lambda_\pi \times V_\pi$
such that $u_1^{-1}v_1 = 
u_2^{-1}v_2$. It is sufficient to show that $u^{-1}\pi \neq \pi^{-1}v$
for all $(u, v) \in \Lambda_\pi \times V_\pi$, and $u, v \neq \pi$. 

Let $r_1(C_\Lambda) = (i_1, i_2, \dots, i_{k_1})$ be the reduced
decomposition of $u^{-1}\pi$ and $r_2(C_V) = (j_1, j_2, \dots,
j_{k_2})$ be the reduced decomposition of $\pi^{-1}v$. We need only
consider the situation when $k_1 = k_2$.  

Since $\pi_A$ is a permutation
of $\{1,2,\dots, m\}$ and $\pi_B$ is a permutation of $\{m+1, \dots,
n\}$, by Lemma \ref{no interaction} we can write $u = u_Au_B$ where
$u_A$ is a permutation of $\{1,2,\dots, m\}$ and $u_B$ is a
permutation of $\{m+1, \dots, n\}$.  Furthermore, we can also write the reduced
decomposition of $u^{-1}\pi$ as a concatenation of the reduced
decompositions of 
$u_A^{-1}\pi_A$ and $u_B^{-1}\pi_B$. Accordingly, if there exists $v
\geq \pi$ such that $\pi v = u^{-1}\pi $, then we can write $v$ as a concatenation of two subpermutations $ 
v_A, v_B$, and the reduced decomposition for $\pi^{-1}u$ is a concatenation of 
$\pi_A^{-1}v_A$ and $\pi_B^{-1}v_B$. Hence in order to
have $u^{-1}\pi = \pi v$, we must have 
 $$u_A^{-1}\pi_A
   =\pi_A^{-1}v_A \text{ 
  and  } u_B^{-1}\pi_B=\pi_B^{-1}v_B.$$ %
Thus we need only consider the case in which the size of the
permutation is less than $n$.

For the surjectivity part, we want to show that, for each permutation
$w \in \sn$, there exists $(u,v)\in \Lambda_\pi \times V_\pi$ such
that $u^{-1}v = w$. 

Let $w\in \mathfrak{S}_n$ be as in Proposition~\ref{above}(ii).  
Let $w_1$ be the sub-permutation of $w$ which consists of the letters
$\{1,2,\dots, m\}$, and let $w_2$ be the sub-permutation of $w$ which
consists of $\{m+1,m+2,\dots, n\}$. 
By the inductive hypothesis, there exist $(u_1, v_1) \in
\Lambda_{\pi_A} \times V_{\pi_A}$  and  $(u_2, v_2) \in
\Lambda_{\pi_B} \times V_{\pi_B}$ such that $$u_1^{-1}v_1 = w_1 \text{
  and } u_2^{-1}v_2= w_2.$$ It follows that $$(u_1u_2)^{-1}(v_1v_2) = w_1w_2,$$      
and $(u_1u_2, v_1v_2) \in \Lambda_{\pi} \times V_{\pi}.$

We now show that we can find $v' \geq v_1v_2$ such that
$$(v_1v_2)^{-1}v' = (w_1w_2)^{-1}w.$$  
Then it follows that for any arbitrary $w$, there exists a $(u_1u_2,
v') \in \Lambda_{\pi} \times V_{\pi}$ and $$(u_1u_2)^{-1}v'
=(u_1u_2)^{-1}(v_1v_2)(v_1v_2)^{-1}v' = (w_1w_2)(w_1w_2)^{-1}w = w.$$
We will show an explicit way to find $v'$. 

Let $A_1 < A_2 < \cdots <A_m$ be the positions in $\pi$ that are
occupied by the letters $\{1,2,\dots, m\}$. We start by shifting the letters
$\{1,2,\dots, m\}$ in both $v_1v_2$ and $w_1w_2$ to the positions
indexed by $A_1, A_2, \dots, A_m$. That is, we move the letters at the
$m$th position in $v_1v_2$ and $w_1w_2$ to the position indexed by
$A_m$, and then move the letter at the $(m-1)$-st position to the
position indexed by $A_{m-1}$, and so on. Finally, we move the letter
at the first position to the position indexed by $A_1$.  
Recall that $v_1$ and $w_1$ are permutations of $\{1,2,\dots, m\}$
and $v_2$ and $w_2$ are permutations of $\{m+1,\dots, n\}$. 
Since $A_1 < A_2 < \cdots < A_m,$ it is easy to show that during the
shifting process, all the transpositions are between a letter in
$\{1, 2, \dots, m\}$ and a letter in $\{m+1, m+2, \dots, n\}$, and
that after each transposition, the length of the permutation increases
by $1$. This process thus turns $w_1w_2$ into $w$ and $v_1v_2$ into
another permutation, which we set to be $v'$. Accordingly, by the inductive
hypothesis and this shifting process, we have an explicit way to find
$v'$ such that $$(u_1u_2)^{-1}v' = w.$$ 

When $a_1 > a_n$, we use the complement of the permutation, and the
rest of the proof is similar. 
\end{proof}


\subsection{Explicit Formulas for $F(\Lambda_\pi, q)$ and\ $F(V_\pi,q)$}
Based on Proposition \ref{below} and Proposition \ref{above}, we
introduce a convenient method to find the explicit formulas for
$F(\Lambda_\pi, q)$ and $F(V_\pi,q)$. 

The most convenient way is to use a \emph{separating tree}. We
define it recursively as follows.

Let $\pi=a_1a_2\cdots a_n$ be a separable permutation. 

When $n=2$, its separating tree $T_\pi$ is an ordered binary tree with
the left leaf $a_1$ and right leaf $a_2$. 

When $n > 2$, by Lemma \ref{separable} we can write $\pi = \pi_A
\pi_B$ where $\pi_A$ and $\pi_B$ are separable permutations with size
strictly smaller than $n$. Then $T_\pi$ is an ordered binary tree,
with the subtree rooted at the left child of the root, being
$T_{\pi_A}$, and the subtree rooted at the right child of the root,
being $T_{\pi_B}$.  

Since there might be more than one  way to write $\pi = \pi_A\pi_B$, a
separable permutation can have more than one separating tree. Also,
only the separable permutations have separating trees. 

The definition of the separating tree $T_\pi$ gives the following
lemma, which is easy to prove. 

\begin{lemma}
For any node in $T_\pi$, the leaves of the subtree rooted at that node
form a subrange, a set of consecutive integers. 
\end{lemma} 

This lemma allows us to classify the nodes in $T_\pi$ into two
categories. A node is \emph{negative} if the subrange of
its left child is greater than that of its right child. \emph{Positive}
 node is defined analogously.  
Figure 3 shows a separating tree for $4231$, which has two negative nodes and one positive node, as labeled in the figure. 
\vspace{0.2cm}
\begin{center}
\scalebox{0.75}{
$$\xymatrix@!C@!R@R=12pt@C=-26pt@M=0pt@H=0pt@W=0pt{
& & & \txt<4pc>{Negative\\Node} \ar@{-}[dddlll] \ar@{-}[dr] & & & \\
& &&  &  \txt<4pc>{Negative\\Node} \ar@{-}[dl] \ar@{-}[ddrr] & &\\
& & & \txt<4pc>{Positive\\Node} \ar@{-}[dl]  \ar@{-}[dr] & & & \\
4 & & 2 & & 3 & & 1  \\
}\label{treeS}$$}
\end{center}
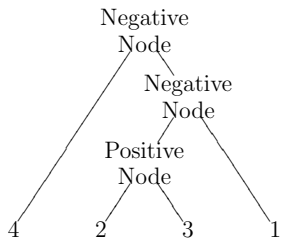
\begin{center}{ Figure 3.  The separating tree for 4231}\end{center}
\vspace{0.3cm}

\begin{theorem}\label{formula}
Let  $S^{-}(\pi) = \{$all negative nodes $V_i$ in $T_\pi$ whose
parents are not negative$\}$ and $S^{+}(\pi)=\{$all positive
nodes $V_j$ in $T_\pi$ whose parents are not positive$\}$. Let 
$V_0$ be the root of the tree, and $V_0$ is not in either $S^{-}(\pi)$ nor $S^{+}(\pi)$. 
Let $N(V_k)$ denote the number of leaves in the subtree rooted at
$V_k$. In particular, we define $\prod_{V_i \in \varnothing}{
  [N(V_i)]!} = 1$. Then 

\begin{eqnarray} \label{formulabelow}
F(\Lambda_w,q) =\left\{\begin{array}{lr}
\dfrac{\prod_{V_i \in S^{-}(\pi)}{[N(V_i)]! }} { \prod_{V_j \in
    S^{+}(\pi)}{ [N(V_j)]!}}, &V_0\text{ is a positive node;} \\[2em] 
\dfrac{\prod_{V_i \in S^{-}(\pi)}{[N(V_i)]! }} { \prod_{V_j \in
    S^{+}(\pi)}{ [N(V_j)]!}}[N(V_0)]!, &V_0\text{ is a negative
  node}. 
\end{array}\right.
\end{eqnarray}

\begin{eqnarray} \label{formulaabove}
F(V_w,q) =\left\{\begin{array}{lr}
\dfrac{\prod_{V_i \in S^{+}(\pi)}{[N(V_i)]! }} { \prod_{V_j \in
    S^{-}(\pi)}{ [N(V_j)]!}}[N(V_0)]!, &V_0\text{ is a positive
  node;}\\[2em]
\dfrac{\prod_{V_i \in S^{+}(\pi)}{[N(V_i)]! }} { \prod_{V_j \in
    S^{-}(\pi)}{ [N(V_j)]!}}, &V_0\text{ is a negative node}. 
\end{array}\right.
\end{eqnarray}

\end{theorem}

\begin{example} Let $w = 4231$. Its separating tree is shown in Figure
  3. It has one negative node with no parent, one negative node with
  a negative parent node, and one positive node with a negative parent
  node. Thus $F(\Lambda_w, q) = [4]! / [2]!$, and $F(V_w, q)
  = [2]!/[1]!$. 
\end{example}

\begin{proof}[Proof of Theorem \ref{formula}]
Let $\pi = a_1a_2\cdots a_n$ be a separable permutation. We can use
induction to prove Theorem \ref{formula}. 

By the definition of $N(V)$, we have $N(V_0) = n$. 

When $a_1 < a_n$, we write $\pi =\pi_A\pi_B$ where $\pi_A$ is a
permutation of $\{1, 2,\dots, m\}$. The root $V$ of $T_\pi$ has two
children with the left child $V_L$ having leaves $\{1,2,\dots, m\}$
and the right child $V_R$ having leaves $\{m+1, m+2, \dots, n\}$. Thus
$V$ is a positive node. Let $T_L$ be the subtree rooted at $V_L$ and
$T_R$ be the subtree rooted at $V_R$. Applying formula
(\ref{formulabelow}) to $\pi_A$ and $\pi_B$, together with
(\ref{belowsb}) and (\ref{abovesb}), we can prove (\ref{formulabelow})
and (\ref{formulaabove}) by induction. 

When $a_1 > a_n$, the root of $T_\pi$ is a negative node. The rest of
the proof is similar to the case above when $a_1 < a_n$. 
\end{proof}
 

More specifically, when the permutation $\pi = a_1a_2\cdots a_n$ is
$231$-avoiding (a 231-avoiding permutation requires more restrictions
than a general separable permutation), a more direct formula for
$F(\Lambda_\pi,q)$ can be given.

\begin{corollary}[explicit formula for $F(\Lambda_\pi,q)$ for a
  231-avoiding permutation]  
\label{explicit 231} Let $\pi = a_1 a_2 \cdots a_n$ be $231$-avoiding,
and $a_{c_i+i}$ be the first element to the right of $a_i$ in $\pi$
satisfying $a_{c_i+i}>a_i$, setting $a_{n+1}=\infty$. Then 
\[ F(\Lambda_\pi,q) = \prod_{i=1}^n [c_i]. \label{eq:rgf}
\]\end{corollary}

Before proving this proposition, we give an example to explain the
notation in the formula. 
\begin{example}
Let $\pi = a_1a_2\cdots a_6=142365$. We set $a_7 = \infty$. For $a_1
=1$, letter $4$ is the first one greater than $1$ and to the right of $a_1$; 
the distance between these two integers, $c_1$, is thus
$2-1=1$. Similarly, $c_2 = 5-2=3, c_3 = 4-3=1, c_4 = 5-4=1, c_5=7-5=2,
c_6=7-6=1$. Thus the generating function is $F(\Lambda_{142365}, q) =
\prod_{i=1}^{6}{[c_{i}] } = [1][3][1][1][2][1] =(q^2+q+1)(q+1)$. 
\end{example}

\begin{proof}
By Lemma \ref{no interaction} we know that when $\pi$ is
$231$-avoiding, either $\pi$ has the greatest letter $n$ at its first
position, or $n$ is at the $(m+1)$-st position with $m>0$.  Thus we
can write $\pi = \pi_A\pi_B$ where $\pi_A$ is a 231-avoiding
permutation of $\{1,\dots, m\}$ and $\pi_B$ is a 231-avoiding
permutation of $\{m+1,\dots, n\}$. Then we can construct the
separating tree by repeatedly applying the following steps.

For a separating tree with root $V_0$, we first decide its left child
$V_L$ and right child $V_R$ by identifying the position of the
greatest letter in $\pi$, i.e., finding $m$ such that $a_{m+1} =
n$.  

When $m=0$, the subtree rooted at $V_L$ has only one leaf $a_1 = n$,
while the subtree rooted at $V_R$ is the separating tree of the
permutation $a_2a_3\cdots a_n$, which we will construct similarly. 

When $m > 0$, the subtree rooted at $V_L$ is the separating tree for
the permutation $a_1a_2\cdots a_m$, while the subtree rooted at $V_R$
is the separating tree for the permutation $a_{m+1}\cdots a_n$. We
then construct these two separating tree similarly.

In the first case, $V_0$ is a negative node. We already know that $c_1=n$ and 
$$F(\Lambda_{\pi}, q) = [n]\cdot F(\Lambda_{a_2a_3\cdots a_n},q),$$ and
$c_1 = n.$  

In the second case, $V_0$ is a positive node. We have $c_1 = n$
and $$F(\Lambda_{\pi}, q) = F(\Lambda_{a_1a_2\cdots a_m},q)
F(\Lambda_{a_{m+1}a_{m+2}\cdots a_n},q).$$   
We also know that, for a letter $a$ in $\{1,2,\dots, m\}$, the
distance between $a$ and the first letter greater than $a$ and to its
right is the same in both $\pi$ and $a_1 a_2\cdots a_m$.  

The rest of the proof can be completed by induction.
\end{proof}

\vskip 0.2cm
Since we know that $F(\Lambda_{\pi}, q)F(V_{\pi}, q) = [n]!$, as well
as the explicit formula for $F(\Lambda_{\pi}, q)$, we can also obtain
an explicit formula for $F(V_{\pi}, q)$. 

By symmetry, we can obtain analogous explicit formulas when the
permutation avoids any of the patterns 132, 231, 312, or 213. 

The following two lemmas are standard results about unimodality; see
for instance \cite{rs:unim}.

\begin{lemma} \label{young}
The $q$-binomial coefficient $\begin{bmatrix} n\\ m \end{bmatrix}$ is 
rank-unimodal and rank-symmetric. 
\end{lemma}

\begin{lemma}  \label{product of unimodal}  
Let $F(q)$ and $G(q)$ be symmetric unimodal polynomials with
nonnegative real coefficients. Then $F(q)G(q)$ is also symmetric and
unimodal. \end{lemma} 

Lemma \ref{young} and Lemma \ref{product of unimodal} imply the following corollary.
\begin{corollary} \label{sym and unim}
$F(\Lambda_\pi, q)$ and $F(V_\pi, q)$ are rank symmetric and
unimodal. 
\end{corollary}

Theorem~\ref{formula} determines the number of elements of each rank
$k$ of the poset $\Lambda_\pi$ when $\pi$ is separable. We can
also determine the number of elements that cover $k$ elements.
A \emph{descent} of a permutation $\pi = a_1a_2\cdots a_n \in \sn$ is a
position $i$ with $1\leq i <n$, such that $a_i > a_{i+1}$. Let
$\mathrm{des}(\pi)$ be the number of descents of $\pi$. 
It is easy to see that $\mathrm{des}(\pi)$ is equal to the number of
elements that $\pi$ covers in the weak order on $\sn$. If $P_\pi$ is the
inversion poset of $\pi$, then the enumeration of linear extensions of
$P_\pi$ by number of descents is the same as the enumeration of elements
of $\Lambda_\pi$ in weak order by number of covers. Let
$\Omega_P(m)$ denote the number of order-preserving maps $f\colon P\to
\{1,\dots,m\}$. Then we have the following theorem which relates
$\Omega_P(m)$ with the descent number. The proof can be found in
\cite[Thm.~4.5.14]{EC1}. 

\begin{theorem} \label{des}
For any poset $P$ on $\{1,2,\dots,n\}$, we have
$$ \sum_{m\geq 1}\Omega_P(m)x^m
=\frac{\sum_{\pi\in\mathcal{L}(P)}
          x^{\mathrm{des}(\pi)+1}}{(1-x)^{n+1}}. $$
\end{theorem}

Using the recursive structure of $P_\pi$ when $\pi$ is separable
(Corollary~\ref{cor:ppi}) we can give a recursive description of
$\Omega_{P_\pi}(m)$ and thus of the number of elements in
$\Lambda_\pi$ that cover $k$ elements. We do not enter into the
details here.

Our results suggest several open problems. For what permutations
$\pi\in\sn$ is the poset $\Lambda_\pi$ rank-symmetric? When is
$[n]!$ divisible by the rank generating function
$F(\Lambda_\pi,q)$? When is $F(\Lambda_\pi,q)$ a product of cyclotomic
polynomials? R. Stanley has verified that for $n\leq 8$, if
$\Lambda_\pi$ is rank-symmetric then $F(\Lambda_\pi,q)$ is a product
of cyclotomic polynomials, but $F(\Lambda_\pi,q)$ need not divide
$[n]!$. For instance, when $n=8$ there are 8558 separable
permutations, 10728 permutations $\pi$ for which $\Lambda_\pi$ is
rank-symmetric (and hence a product of cyclotomic polynomials), and
961 permutations $\pi$ for which $\Lambda_\pi$ is rank-symmetric but
$F(\Lambda_\pi,q)$ does not divide $[8]!$. A further problem is to
extend our work to the weak order of other Coxeter groups.



\begin{thebibliography}{9}
\bibitem{albert} M. H. Albert, Aspects of separability,
  
  $\langle$\texttt{http://www.cs.otago.ac.nz/staffpriv/malbert/Talks/Sep.pdf}$\rangle$. 

\bibitem{av-ne} D. Avis and M. Newborn, On pop-stacks in series,
  \emph{Utilitas Math.}\ \textbf{19} (1981), 129--140.

\bibitem{bj-wa} A. Bj\"orner and M. Wachs, Permutation statistics and
  linear extensions of posets, \emph{J. Combinatorial Theory, Ser.~A}\
  \textbf{58} (1991), 85--114.

\bibitem{bona} M. Bona, \emph{Combinatorics of Permutations}, Chapman
  Hall-CRC, Boca Raton, FL, 2004.

\bibitem{tree} P. Bose, J. Buss, A. Lubiw, Pattern matching for
  permutations, \emph{Information Processing Letters}, 
  \textbf{65} (1998), 277--283.

\bibitem{rs:unim} R. Stanley, Unimodal and log-concave sequences in
  algebra, combinatorics, and eometry, in \emph{Graph Theory and Its
    Applications: East and West}, Ann.\ New York Acad.\ Sci., vol.\
  576, 1989, pp.\ 500--535.




\bibitem{EC1} R.Stanley, \emph{Enumerative Combinatorics}, 
vol.~I, Cambridge University Press, Cambridge/New York, 1997.


\end{thebibliography}
\end{document}